\documentclass[12pt,norsk]{amsart}
\usepackage[T1]{fontenc}
\usepackage[utf8]{inputenc}
\usepackage{amsmath}
\usepackage{fourier}
\usepackage{amsfonts}
\usepackage{amssymb}
\usepackage{times,a4wide}

\newtheorem{theorem}{Theorem}
\newtheorem{lemma}[theorem]{Lemma}

\newtheorem*{question}{Question}

\newcommand{\T}{\mathbb T}

\newenvironment{proof*}{\textbf{\emph{Proof of }}}{\hspace*{\fill} $\square$}
\begin{document}
\title[Helson's problem for sums of a random multiplicative function]
{Helson's problem for sums of a random multiplicative function}

\author{Andriy Bondarenko}
\address{Department of Mathematical Analysis\\ Taras Shevchenko National University of Kyiv\\
Volodymyrska 64\\ 01033 Kyiv\\ Ukraine}
\address{Department of Mathematical Sciences \\ Norwegian University of Science and Technology \\ NO-7491 Trondheim \\ Norway}

\email{andriybond@gmail.com}
\author[Kristian Seip]{Kristian Seip}
\address{Department of Mathematical Sciences \\ Norwegian University of Science and Technology \\ NO-7491 Trondheim \\ Norway}
\email{seip@math.ntnu.no}

\thanks{Research supported by Grant 227768 of the Research Council of Norway.}
\maketitle

\begin{abstract} We consider the random functions $S_N(z):=\sum_{n=1}^N z(n) $, where $z(n)$ is the completely multiplicative random function generated by independent Steinhaus variables $z(p)$. It is shown that ${\Bbb E} |S_N|\gg \sqrt{N}(\log N)^{-0.05616}$ and that $({\Bbb E} |S_N|^q)^{1/q}\gg_{q} \sqrt{N}(\log N)^{-0.07672}$ for all $q>0$. 
\end{abstract}

\section{Introduction}

This paper deals with the following

\begin{question}
Do there exist absolute constants $c>0$, $ 0<\lambda <1$ such that for every positive integer $N$ and every interval $I$ whose length exceeds some number depending on $N$, we have
 \[\left| \sum_{n=1}^N n^{-it} \right|\ge c\sqrt{N}\] on a subset of $I$ of measure larger than $\lambda |I|$?
\end{question}
We do not know the answer and can only conclude from our main result that we have, for every $\varepsilon>0$ and suitable $c=c(\varepsilon)$,
\begin{equation} \label{almost} \left| \sum_{n=1}^N n^{-it} \right|\ge c\sqrt{N}(\log N)^{-0.07672}\end{equation}
on a subset of measure $(\log N )^{-\varepsilon} |I|$ of every sufficiently large interval $I$.

Our question fits into the following general framework.  We begin by associating with every prime $p$ a random variable $X(p)$ with mean $0$ and variance $1$, and we assume that these variables are  independent and identically distributed. We then define $X(n)$ by requiring it to be a completely multiplicative function for every point in our probability space. Now suppose that $a(n)$ is an arithmetic function which is either $0$ or $1$ for every $n$. We refer to the sequence 
\[ C_N(X):=\sum_{n=1}^N a(n) X(n) \]
as the arithmetic chaos associated with $X$ and $a(n)$. 

Our question concerns the case when $X(p)$ are independent Steinhaus variables $z(p)$, i.e. the random variable $z(p)$ is equidistributed on the unit circle. When $a(n)\equiv 1$, we refer to the resulting sequence
\[ S_N(z):=\sum_{n=1}^N  z(n) \] 
as arithmetic Steinhaus chaos. The relation between our question  and arithmetic Steinhaus chaos is given by the well-known norm identity
\begin{equation} \label{normid} {\mathbb E} (|S_N|^q) =\lim_{T\to \infty} \frac{1}{T} \int_0^T \left|\sum_{n=1}^N n^{-it}\right|^q dt, \end{equation}
valid for all $q>0$ (see \cite[Section 3]{SS}).

The point of departure for our research is Helson's last paper \cite{H} in which he conjectured that ${\Bbb E} (|S_N|)=o(\sqrt{N})$ when $N\to \infty$. This means that Helson anticipated that our question has a negative answer. Using an inequality from another paper of Helson \cite{H2}, we get immediately that
\begin{equation} \label{henry} {\Bbb E} (|S_N|)\gg \sqrt{N} (\log N)^{-1/4}. \end{equation}
Our attempt to settle Helson's conjecture has resulted in a reduction from $1/4$ to $0.05616$ in the exponent of the logarithmic factor in  \eqref{henry}. We note in passing that the problem leading Helson to his conjecture was solved in \cite{OS} avoiding the use of the random functions $S_N$.

To get a picture of what our work is about, it is instructive to return 
for a moment to a general arithmetic chaos $C:=(C_N(X))$. To this end, let us assume that $X(p)$ is such that the moments
\[ \| C_N \|_q^q:={\Bbb E}(|C_N|^q) \]
are well defined for all $q>0$. We declare the number
\[ q(C):=\inf\left\{q>0: \ \limsup_{N\to \infty} \| C_N \|_{q+\varepsilon}/\|C_N \|_q=\infty \quad \text{for every $\varepsilon>0$}\right\} \]
to be the critical exponent of $C$, setting $q(C)=\infty$ should the set on the right-hand side be empty. A problem closely related to Helson's conjecture is that of computing the critical exponent of a given arithmetic chaos. We observe that $q(C)\ge 2$ is equivalent to the statement that there exist absolute constants $c>0$, $0<\lambda<1$ such that
\[ {\Bbb P}\left(|C_N|\ge c\sqrt{N}\right)\ge \lambda \]
holds for all $N$, cf. our question. In our case, the critical exponent is strictly smaller than 4, and then a serious obstacle for saying much more is that only even moments are accessible by direct methods. 

We will prove the following result about  arithmetic Steinhaus chaos.

\begin{theorem}\label{th18}
We have
\begin{equation} \label{small} \| S_N\|_q \gg_q \sqrt{N}(\log N)^{-0.07672} \end{equation}
for all $q>0$. 
\end{theorem}
This  estimate is of course of interest only for small $q$; our method allows us to improve \eqref{small} for each individual $0<q<2$ as will be demonstrated in the last section of the paper. In the range $q>2$, we note that the $L^4$ norm has an interesting number theoretic interpretation and has been estimated with high precision \cite{ACZ}:
\[ \left\| S_N\right\|_4^4=\frac{12}{\pi^2} N^2 \log N + c N^2 + O\left(N^{19/13} (\log N)^{7/13}\right)\]
with $c$ a certain number theoretic constant. This means in particular that the critical exponent of arithmetic Steinhaus chaos $S:=(S_N)$ satisfies $q(S)<4$. We mention without proof that, applying the Hardy--Littlewood inequality from \cite{BHS} to $S_N^2$, we have been able to verify that in fact $q(S)\le 8/3$. A further elaboration of our methods could probably lower this estimate slightly, but this would not alter the main conclusion that it remains unknown whether $q(S)$ is positive. 

Before turning to the proof of Theorem~\ref{th18}, we mention the following simple fact: There exists a constant $c<1$ such that $\| S_N \|_1\le c \| S_N\|_2$ when $N\ge 2$. To see this, we apply the Cauchy--Schwarz inequality to the product of $(1-\varepsilon z(2))S_N$ and $(1-\varepsilon z(2))^{-1}$ to obtain 
\[ \| S_N\|_1^2\le  \frac{1}{(1-\varepsilon^2)}\cdot  \left((1-\varepsilon)^2 [N/2]+(1+\varepsilon^2)[(N+1)/2]\right)\le \frac{N-(\varepsilon-\varepsilon^2)(N-1)}{1-\varepsilon^2} \]
for every $0<\varepsilon<1$. Choosing a suitable small $\varepsilon$, we obtain the desired constant $c<1$.

\section{Proof of Theorem~\ref{th18}}

Our proof starts from a decomposition of $S_N$ into a sum of homogeneous polynomials. To this end, we set 
\[ E_{N,m}:=\left\{n\le N: \ \Omega(n)=m\right\}, \]
where $\Omega(n)$ is the number of prime factors of $n$, counting multiplicities. Correspondingly, we introduce the homogeneous polynomials
\[ S_{N,m}(z):=\sum_{n\in E_{N,m}} z(n) \]
so that we may write
\begin{equation} \label{homo} S_{N}(z)=\sum_{m\le(\log N)/\log 2} S_{N,m}(z). \end{equation}

We need two lemmas. The first is a well-known estimate of Sathe; the standard reference for this result is Selberg's paper \cite{S}. To formulate this lemma, we introduce the  function
\[ \Phi(z):=\frac{1}{\Gamma(z+1)}\prod_{p} \left(1-1/p)\right)^z \left(1-z/p\right)^{-1}, \]
where the product runs over all prime numbers $p$. This function is meromorphic in ${\Bbb C}$ with simple poles at the primes and zeros at the negative integers.

\begin{lemma}\label{sathe}
When $N\ge 3$ and $1\le m \le (2-\varepsilon) \log\log N$ for $0<\varepsilon<1$, we have
\[ |E_{N,m}|=\frac{N}{\log N} \Phi\left(\frac{m}{\log\log N}\right) \frac{(\log\log N)^{m-1}}{(m-1)!}\left(1+O\left(\frac{1}{\log\log N}\right)\right), \]
where the implied constant in the error term only depends on $\varepsilon$.
\end{lemma}

The second lemma is a general statement about the decomposition of a holomorphic function into a sum of homogeneous polynomials. For simplicity, we consider only an arbitrary
holomorphic polynomial $P(z)$ in $d$ complex variables $z=(z_1,...,z_d)$. Such a polynomial has a unique decomposition
\[ P(z)=\sum_{m=0}^k P_m(z), \]
where $k$ is the degree of $P$ and 
\[ P_m(z)=\sum_{|\alpha|=m} a_{\alpha} z^\alpha \]
is a homogeneous polynomial of degree $m$. Here we use standard multi-index notation, which means that $\alpha=(\alpha_1,..., \alpha_d)$, where $\alpha_1$, ..., $\alpha_d$ are nonnegative integers,
\[ z^{\alpha}=z_1^{\alpha_1}\cdots z_d^{\alpha_d}, \]
and $|\alpha|=\alpha_1+\cdots+\alpha_d $. At this point, the reader should recognize that if we represent an arbitrary integer $n\le N$ by its prime factorization $p_1^{\alpha_1}\cdots p_d^{\alpha_d}$ (here $d=\pi(N)$) and set $\alpha(n)=(\alpha_1,...,\alpha_d)$, then we may write 
\[ S_N(z)=\sum_{n=1}^N z^{\alpha(n)}. \]
Hence, as already pointed out, \eqref{homo} is the decomposition of $S_N$ into a sum of homogeneous polynomials, and we also see that $|\alpha(n)|=\Omega(n)$.

We let $\mu_d$ denote normalized Lebesgue measure on $\T^d$ and define
\[ \| P\|_q^q:= \int_{\T^d} |P(z)|^q d\mu_d(z)\]
for every $q>0$. The variables $z_1$,..., $z_d$ can be viewed as independent Steinhaus variables so that $\| S_N\|_q$ has the same meaning as before.
\begin{lemma}\label{coeff}
There exists an absolute constant $C$, independent of $d$, such that
\[ \| P_m \|_q \le \begin{cases} \| P \|_q, & q\ge 1 \\
                                               C m^{1/q-1} \| P \|_q, & 0<q<1\end{cases} \]
holds for every holomorphic polynomial $P$ of $d$ complex variables.
\end{lemma}

\begin{proof}
We introduce the transformation $z_w=(wz_1,...,wz_d)$, where $w$ is a point on the unit circle $\T$. We may then write
\[ P(z_w)=\sum_{m=0}^k P_m(z) w^m.\]
It follows that we may consider the polynomials $P_m(z)$ as the coefficients of a polynomial in one complex variable. Then a classical coefficient estimate
(see \cite[p. 98]{D}) shows that
\[ |P_m(z)|^q\le \begin{cases} \int_{\T} |P_m(z_w)|^q d\mu_1(w), & q\ge 1 \\
                                               C m^{1-q} \int_{\T} |P_m(z_w)|^q d\mu_1(w), & 0<q<1.\end{cases} \]
Integrating this inequality over $\T^d$ with respect to $d\mu_d(z)$ and using Fubini's theorem, we obtain the desired estimate.
\end{proof}

We now turn to the proof of Theorem~\ref{th18}. The idea of the proof can be related to an interesting study of Harper \cite{Ha} from which it can be deduced that, asymptotically, the square-free part of the homogeneous polynomial $S_{N,m}/\sqrt{N}$  has a Gaussian distribution when $m=o(\log\log N)$. When $m=\beta \log\log N$ for $\beta$ bounded away from $0$, this is no longer so, but what we will use, is a much weaker statement: When $\beta$ is small enough, the $L^2$ and $L^4$ norms are comparable. The proof will consist in identifying for which 
$\beta$ this holds. 

To this end, we first observe that
\begin{equation}\label{l2}
\| S_{N,m}\|_2^2 =|E_{N,m}| .
\end{equation}
To estimate $\| S_{N,m} \|_4^4$, we begin by noting that
\[ |S_{N,m}|^2=|E_{N,m}|+ \sum_{k=0}^{m} \sum_{a,b\in E_{N,k}, (a,b)=1} |E_{N/\max(a,b),m-k}| z(a)\overline{z(b)}. \]
Squaring this expression and taking expectation, we obtain
\begin{align} \|S_{N,m}\|_4^4 & =|E_{N,m}|^2+2 \sum_{k=0}^{m} \sum_{a,b\in E_{N,k}, (a,b)=1, a<b} |E_{N/b,m-k}|^2 \nonumber  \\
& \le |E_{N,m}|^2+ 2 \sum_{k=0}^{m}\sum_{b\in E_{N,k}} |E_{b, k}|\cdot |E_{N/b,m-k}|^2  \nonumber \\
& \le 5 |E_{N,m}|^2+ \sum_{k=1}^{m-1}\sum_{b\in E_{N,k}} |E_{b, k}|\cdot |E_{N/b,m-k}|^2. \label{last} \end{align}
Here we used that, plainly,
\[ \sum_{b\in E_{N,0}} |E_{N/b,m}|^2=|E_{N,m}|^2 \quad \text{and} \quad \sum_{b\in E_{N,m}} |E_{b,m}| \le |E_{N,m}|^2  . \] 
To estimate the sum over $b$ in \eqref{last}, we begin by observing that Lemma~\ref{sathe} 
implies that
\begin{align} |E_{N/b,m-k}| & \ll b^{-1} |E_{N,m-k}|, \quad b\le \sqrt{N},  \nonumber \\
|E_{b,k}| & \ll bN^{-1} |E_{N,k}|, \quad \sqrt{N}<b\le N. \label{largeN}
\end{align}
We split correspondingly the sum into two parts:
\begin{equation} 
\sum_{b\in E_{N,k}} |E_{b, k}|\cdot |E_{N/b,m-k}|^2 
\ll |E_{N,m-k}|^2 \sum_{b\in E_{\sqrt{N}, k}} b^{-2} |E_{b, k}| + |E_{N,k}| \sum_{b\in E_{N,k}\setminus E_{\sqrt{N}, k}} b N^{-1} |E_{N/b,m-k}|^2.
\label{split} \end{equation}
To deal with the first of the two sums in \eqref{split}, we begin by using Lemma~\ref{sathe} so that we get smooth terms in the sum:
\[ \sum_{b\in E_{\sqrt{N},k}} b^{-2} |E_{b,k}| \ll \sum_{b\in E_{\sqrt{N},k}\setminus\{1,2\}} \frac{(\log\log b)^{k-1}}{b (\log b) (k-1)!}=
\sum_{2<b\le \sqrt{N}} g(b) \frac{(\log\log b)^{k-1}}{b (\log b) (k-1)!}, \]
where $g(n)$ is the characteristic function of the set $E_{\sqrt{N},k}$. We apply Abel's summation formula to the latter sum and  obtain, using  also Lemma~\ref{sathe},
\begin{align} \sum_{b\in E_{\sqrt{N},k}} b^{-2} |E_{b,k}| 
 & \ll
\big| E_{\sqrt{N},k}\big| N^{-1/2}\frac{(\log\log N)^{k-1}}
{(\log N)(k-1)!}+\int_3^{N} \big| E_{x,k}\big| \frac{(\log\log x)^{k-1}}
{x^2(\log x)(k-1)!} dx \nonumber \\ 
& \ll \frac{(\log\log N)^{2(k-1)}}
{(\log N)^2((k-1)!)^2}+\int_3^{N}  \frac{(\log\log x)^{2(k-1)}}
{x(\log x)^2((k-1)!)^2} dx  \nonumber \\
& \ll \frac{1}{((k-1)!)^2}\int_0^{\infty}  y^{2(k-1)}
e^{-y} dy = \frac{(2k-2)!}{((k-1)!)^2}\ll \frac{2^{2k}}{\sqrt{k}}. \label{small2}
 \end{align}
Arguing in a similar fashion, using Abel's summation formula and again \eqref{largeN}, we get
\begin{align} \sum_{b\in E_{N,k}\setminus E_{\sqrt{N},k}} bN^{-1} |E_{N/b,m-k}|^2 
& \ll 
N \int_{\sqrt{N}}^{N/3} \big| E_{x,k}\big| \frac{\Big(\log\log \frac{N}{x}\Big)^{2(m-k-1)}}
{x^2\big(\log \frac{N}{x}\Big)^2((m-k-1)!)^2} dx \nonumber \\  
& \ll \frac{|E_{N,k}|}{((m-k-1)!)^2}
 \int_0^{\infty}  y^{2(m-k-1)}
e^{-y} dy \nonumber \\
& \ll  |E_{N,k}|
\cdot \frac{2^{2(m-k)}}{\sqrt{m-k}}. \label{big} \end{align}
Inserting \eqref{small2} and \eqref{big} into \eqref{split}, we obtain
 \[ \sum_{b\in E_{N,k}} |E_{b, k}|\cdot |E_{N/b,m-k}|^2  \ll |E_{N,m-k}|^2\cdot \frac{2^{2k}}{\sqrt{k}}+|E_{N,k}|^2\cdot\frac{2^{2(m-k)}}{\sqrt{m-k}}.\]
Returning to \eqref{last} and using \eqref{l2}, we therefore find that
\[ \|S_{N,m} \|_4^4 \ll \| S_{N,m}\|_2^4 \Big(1+\sum_{k=1}^{m-1} \frac{|E_{N,m-k}|^2}{|E_{N,m}|^2}\cdot \frac{2^{2k}}{\sqrt{k}}\Big). \label{final}\]
Applying again Lemma~\ref{sathe}, we get
\begin{align} \sum_{k=1}^{m-1} \frac{|E_{N,m-k}|^2}{|E_{N,m}|^2}\cdot \frac{2^{2k}}{\sqrt{k}} & \ll \sum_{k=1}^{m-1} \frac{1}{\sqrt{k}}\cdot\left(\frac{(m-1)!}{(m-k-1)!}\right)^2\cdot
\left(\frac{2}{\log\log N}\right)^{2k} \nonumber  \\ 
& \ll  \label{improve} \sum_{k=1}^{m-1} \frac{1}{\sqrt{k}}\cdot e^{-\frac{k^2}{m}}\cdot
\left(\frac{2m}{\log\log N}\right)^{2k} .\end{align}
It follows that the two norms are comparable whenever $m=\frac{e^{-\varepsilon}}{2} \log\log N$ for $\varepsilon>0$, in which case H\"{o}lder's inequality yields
\begin{equation}\label{holder} \| S_{N,m} \|_2 \ll_\varepsilon \| S_{N,m} \|_q \end{equation}
for $0<q<2$. By \eqref{l2}, Lemma~\ref{sathe}, and Stirling's formula, we have   
\begin{equation} \label{stirling} \| S_{N,m}\|_2 =|E_{N,m}|^{1/2} \asymp \sqrt{N} (\log N)^{-\delta(\varepsilon)} m^{-1/4} \end{equation}
when $m\le \frac{e^{-\varepsilon}}{2}\log\log N$, where 
\[ \delta(\varepsilon):=(2-e^{-\varepsilon}(1+\log 2+\varepsilon))/4=(1-\log 2)/4+O(\varepsilon)\]
when $\varepsilon\to 0$. Combining this with \eqref{holder} and applying Lemma~\ref{coeff}, we infer that
\[ \sqrt{N} (\log N)^{-\delta(\varepsilon)} (\log\log N)^{-1/4} \ll \| S_{N,m} \|_q \ll (\log\log N)^{\max(1/q-1,0)} \| S_N\|_q \]
when $0<q<2$. Theorem~\ref{th18} follows since $(1-\log 2)/4<0.07672$. 

\section{Concluding remarks}

\noindent \textbf{1.} We will now deduce \eqref{almost} from Theorem~\ref{th18}. Since $t\mapsto \sum_{n=1}^N n^{-it}$ is an almost periodic function, it suffices to consider the interval $I=[0, T]$ for some large $T$. Moreover, by \eqref{normid}, it amounts to the same to estimate the measure of the subset
\[ \mathcal{E}:=\left \{z: \ |S_N(z)|\ge c \sqrt{N}(\log N)^{-0.07672} \right\}\]
of $\T^{\pi(N)}$ for a suitable $c$ depending on $\varepsilon$. We find that
\begin{align*} \| S_N\|_q^q & \le c^q N^{q/2}(\log N)^{-0.07672 q} +\int_{\mathcal E} |S_N(z)|^q d\mu_{\pi(N)}(z) \\
& \le c^q N^{q/2} (\log N)^{-0.07672 q}+\| S_N\|_2^{q} |{\mathcal E}|^{1-q/2},  \end{align*}
where we in the last step used H\"{o}lder's inequality. Using Theorem~\ref{th18} to estimate $\|S_N\|_q^q$ from below and recalling that $\|S_N\|_2=\sqrt{N}$, we therefore get
\begin{equation} \kappa_q (\log N)^{-0.07672q} \le c^q (\log N)^{-0.07672 q}+ |{\mathcal E}|^{1-q/2},\label{meas} \end{equation}
where $\kappa_q$ is a constant depending on $q$.
Given $\varepsilon>0$, we now choose $q$ such that $\varepsilon=0.07672 q/(1-q/2)$ and $c^q=\kappa_q/2$. Then  
\eqref{meas}  yields
\[ |{\mathcal E}|\ge (\kappa_q/2)^{(1-q/2)^{-1}} (\log N)^{-\varepsilon}. \] 

\noindent \textbf{2.} We may improve \eqref{small} in the following way. If $m=\frac{e^y}{2}\log\log N$ with $y>0$, then we see from \eqref{improve} that
\[ \sum_{k=1}^{m-1} \frac{|E_{N,m-k}|^2}{|E_{N,m}|^2} \cdot \frac{2^{2k}}{\sqrt{k}} \ll_y e^{my^2}\]
which in turn implies that $\| S_{N,m} \|_2/\|S_{N,m}\|_4 \gg_y e^{-my^2/4}$. By H\"{o}lder's inequality, 
\[  \| S_{N,m}\|_2 \le \|S_{N,m}\|_q^{\frac{q}{4-q}} \| S_{N,m} \|_4^{\frac{4-2q}{4-q}}, \]
and we therefore get
\[  \| S_{N,m} \|_q \gg_y \| S_{N,m} \|_2 e^{-my^2(2/q-1)/2}= \| S_{N,m} \|_2 (\log N)^{-e^y y^2(2/q-1)/4}. \]
We also observe that \eqref{stirling} now takes the form  
\[ \| S_{N,m}\|_2=|E_{N,m}|^{1/2}\asymp \sqrt{N} (\log N)^{(e^y(1+\log 2-y)-2)/4} m^{-1/4}.\]
We find that the exponent of $\log N$ in the lower bound for $\| S_{N,m}\|_q$ becomes optimal if we choose $y$ as the positive solution to the quadratic equation $(2/q-1)y^2+(4/q-1)y-\log 2=0$; when $q=1$, we get for instance $y=0.21556...$ and hence, after a numerical calculation, 
\[ \| S_N \|_1\gg \sqrt{N}(\log N)^{-0.05616}. \]

\noindent \textbf{3.} Our proof shows that we essentially need $\beta\le 1/2$ for the projection
\[ P_\beta S_N:=\sum_{m\le \beta \log\log N} S_{N,m} \]
to have comparable $L^2$ and $L^4$ norms. To use our method of proof to show that Helson's conjecture fails, it would suffice to know that the projection $P_1 S_N$ has comparable $L^2$ and $L^q$ norms for some $q>2$, because in that case $\|P_1 S_N\|_2\ge (1+o(1))\sqrt{N}$. However, we see no reason to expect that such a $q$ exists.
\vspace{2mm}
  
\noindent \textbf{4.} A careful examination of our proof, including a detailed estimation of the last sum in \eqref{improve}, shows that
\begin{equation} \label{fourtwo} \| S_{N,m} \|_4/\| S_{N,m} \|_2 \asymp (\log\log N)^{1/16} \end{equation}
when $m=\frac{1}{2}\log\log N+O(\sqrt{\log\log N})$. This means that  $\frac{1}{2}\log\log N$ is indeed the critical degree of homogeneity and, moreover, that the two norms fail to be comparable in the limiting case.
\vspace{2mm}

\noindent \textbf{5.} Helson's problem makes sense for other distributions; an interesting case is when
$X(p)$ are independent Rademacher functions $\epsilon(p)$ taking values $+1$ and $-1$ each with probability  $1/2$. If we set $a(n)=|\mu(n)|$ (here $\mu(n)$ is the M\"{o}bius function),  then we obtain  arithmetic Rademacher chaos:
\[ R_N(\epsilon):=\sum_{n=1}^N |\mu(n)| \epsilon(n). \] 
Rademacher chaos was first considered by Wintner \cite{Win} and has been studied by many authors, see e.g. \cite{Hal, Ha}. 
Here it is of interest to note that Chatterjee and Soundararajan showed that $R_{N+y}-R_N$ is approximately Gaussian when $y=o( N/\log N)$ \cite{CS}, which means that the analogue of Helson's conjecture is false in short intervals $[N, N+y]$.
\vspace{2mm}

\noindent \textbf{6.} While we were preparing a revision of this paper, further progress on Helson's problem was announced by Harper, Nikeghbali, and Radziwi{\l\l} \cite{HNR}. By a completely different method, relying on Harper's lower bounds for sums of random multiplicative functions  \cite{Ha2}, these authors obtained the lower bound $\sqrt{N}(\log\log N)^{-3+o(1)}$ for both $\mathbb{E}|R_N|$ and $\mathbb{E}|S_N|$. In view of this result, it seems reasonable to conjecture that $\| S_{N,m} \|_2/\| S_{N,m} \|_1$ is bounded whenever $m=e^{-\varepsilon}\log\log N$ for $\varepsilon>0$ and that $m=\log\log N$ is the limiting case for the boundedness of this ratio. Comparing with \eqref{fourtwo} and taking into account Remark 3 above, one might wonder if the ratio $\| S_{N} \|_2/\| S_{N} \|_1$ does indeed grow as a power of $\log\log N$.

\section*{Acknowledgement}
We are most grateful to Kannan Soundararajan for pertinent and helpful remarks.

\end{document}